\documentclass{amsart}

\usepackage{amsthm,amsmath,amsfonts}
\usepackage{mathdots}
\usepackage{graphicx}
\usepackage{bm}
\usepackage{mathtools}
\usepackage{latexsym}
\usepackage{dsfont}
\usepackage{bbm}
\usepackage{amssymb}
\usepackage{amsmath}
\usepackage{soul}
\usepackage{color}

\newcommand{\toas}{\overset{a.s.}{\underset{n\to\infty}\longrightarrow}}

\def\LCM{\textsl{LCM}\hspace{1pt}}
\def\GCD{\textsl{GCD}\hspace{1pt}}
\def\iprob{\stackrel{\Prob}{\to}}
\def\Prob{\mathbb{P}}

\newcommand{\R}{\mathbb{R}}
\newcommand{\E}{\mathbb{E}}
\newcommand{\N}{\mathbb{N}}
\newcommand{\1}{\mathbbm{1}}
\renewcommand{\P}{\mathbb{P}}

\newtheorem{theorem}{Theorem}[section]
\newtheorem{lemma}[theorem]{Lemma}
\newtheorem{corollary}[theorem]{Corollary}
\newtheorem{proposition}[theorem]{Proposition}

\theoremstyle{remark}
\newtheorem{remark}[theorem]{Remark}

\numberwithin{equation}{section}

\begin{document}

%\title[The LCM of a random set of integers]{The least common multiple of a random set of integers}
\title[Limit theorems for the LCM of a random set of integers]{Limit theorems for the least common multiple of a random set of integers}

\author{Gerold Alsmeyer}
\address{Inst.~Math.~Stochastics, Department of Mathematics and Computer Science, University of M\"unster, Orl\'eans-Ring 10, D-48149, M\"unster, Germany}
\email{gerolda@uni-muenster.de}

\author{Zakhar Kabluchko}
\address{Inst.~Math.~Stochastics, Department of Mathematics and Computer Science, University of M\"unster, Orl\'eans-Ring 10, D-48149, M\"unster, Germany}
\email{zakhar.kabluchko@uni-muenster.de}

\author{Alexander Marynych}
\address{Faculty of Computer Science and Cybernetics, Taras Shevchenko National University of Kyiv, 01601 Kyiv, Ukraine}
\email{marynych@unicyb.kiev.ua}
\thanks{GA and ZK were partially supported by the Deutsche Forschungsgemeinschaft (SFB 878), AM was partially supported by the Alexander von Humboldt Foundation.}


\begin{abstract}
Let $L_{n}$ be the least common multiple of a random set of integers obtained from $\{1,\ldots,n\}$ by retaining each element with probability $\theta\in (0,1)$ independently of the others. We prove that the process $(\log L_{\lfloor nt\rfloor})_{t\in [0,1]}$, after centering and normalization, converges weakly to a certain Gaussian process that is not Brownian motion. Further results include a strong law of large numbers for $\log L_{n}$ as well as Poisson limit theorems in regimes when $\theta$ depends on $n$ in an appropriate way.
\end{abstract}

\maketitle

\section{Introduction and main results}

For $n\in\N$, let $[n]$ denote the set $\{1,2,\ldots,n\}$. Fixing a number $0<\theta<1$, remove each element in $[n]$ with probability $1-\theta$, independently of all other elements in the set. Denote by $A_{n}$ the random subset of remaining elements and by $L_{n}:=\LCM(A_{n})$ their least common multiple. In a recent article, Cilleruelo et al. \cite[Thm.~1.1]{CillRueSarka:14} proved the following weak law of large numbers: As $n\to\infty$,
\begin{equation}\label{eq:WLLN L_n}
\frac{\log L_{n}}{n}\ \iprob\ \frac{\theta\log(1/\theta)}{1-\theta},
\end{equation}
where $\iprob$ means convergence in probability. The result remains valid in the limiting case $\theta=1$ when defining the right-hand side of \eqref{eq:WLLN L_n} as $1$ as well, thus
$$ \lim_{n\to\infty} \frac{\log\LCM([n])}{n} = 1. $$
This is in fact a well-known consequence of the prime number theorem.

On the other hand, the derivation of results beyond \eqref{eq:WLLN L_n}, like a strong law of large numbers or a central limit theorem for $\log L_{n}$, seem to be open problems to our best knowledge. The purpose of this article is to not only provide limit theorems of this kind for both fixed $\theta$ and when $\theta$ varies with $n$, but also prove a functional limit theorem for the stochastic process
$$ t\ \mapsto\ L_{\lfloor nt\rfloor},\quad t\in[0,1] $$
as $n\to\infty$. This latter result will actually be presented first and then yield a central limit theorem for $\log L_{n}$ as an immediate consequence (Corollary \ref{cor:CLT L_n}).

\subsection{A functional central limit theorem}
In order to state the main result, we define the function
\begin{equation}\label{eq:g_def}
g(z)\ :=\ \sum_{k\ge 1}\frac{z^{k}(1-z^{k})}{k(k+1)}
\end{equation}
for $|z|<1$. It can be provided in closed form which is done in Remark \ref{rem:g in closed form} below.

\begin{theorem}\label{thm:main}
As $n\to\infty$, the following weak convergence holds true in the Skorokhod space $D[0,1]$ of c\`adl\`ag functions endowed with the $J_{1}$-topology:
\begin{equation}\label{eq:main_convergence}
\left(\frac{\log L_{\lfloor nt\rfloor}-\E\log L_{\lfloor nt\rfloor}}{\sqrt{n\log n}}\right)_{t\in [0,1]}\ \stackrel{J_{1}}{\Longrightarrow}\ (G(t))_{t\in[0,1]},
\end{equation}
where $(G(t))_{t\in[0,1]}$ is a centered Gaussian process with covariance function
\begin{equation}\label{eq:G_covar2}
\E [G(t)G(s)]= \sum_{k\ge 1}\left(\frac{t}{k}\wedge \frac{s}{k}\right) p_{k} - \sum_{k,l\ge 1}\left(\frac{t}{k}\wedge \frac{s}{l}\right)p_{k}p_{l}
\end{equation}
for $0\le s,t\le 1$, where $p_{k}:=\theta(1-\theta)^{k-1}$ for $k\in\N$. In particular,
\begin{equation}\label{eq:G_var}
{\rm Var}[G(t)]=g(1-\theta)t.
\end{equation}
The process $(G(t))_{t\in[0,1]}$ a.s. has continuous paths.
\end{theorem}

A distributional property as well as a probabilistic representation of the limit process $(G(t))_{t\in[0,1]}$ are given in the subsequent proposition.

\begin{proposition}\label{prop:G_process}
(a) Let $\mathcal{G}_{\theta}$ be a random variable with geometric distribution on $\N$, viz.
$$ \P\{\mathcal{G}_{\theta}=k\}\ =\ p_{k}\ =\ \theta(1-\theta)^{k-1},\quad k\in\N, $$
and $B = (B(t))_{t\in [0,1]}$ an independent standard Brownian motion. Then
\begin{equation}\label{eq:G_rep}
\left(G(t)+\E [B(t/\mathcal{G}_{\theta})|B] \right)_{t\in[0,1]}\ \overset{d}{=}\ \left(B(t\,\E\mathcal{G}_{\theta}^{-1})\right)_{t\in[0,1]},
\end{equation}
where $\E[\cdot|B]$ denotes the conditional expectation and the process $(G(t))_{t\in [0,1]}$ is independent of $(B(t))_{t\in [0,1]},\mathcal{G}_{\theta}$ on the left-hand side.

\vspace{.2cm}
(b) If $B_{1},B_{2},\ldots$ denote independent standard Brownian motions, then
\begin{align}
\begin{split}\label{eq:g_series}
(G(t))_{t\in[0,1]}\ &\overset{d}{=}\ \left(\sqrt{\theta(1-\theta)}\sum_{i\ge 1} (1-\theta)^{(i-1)/2}\right.\\
&\quad\left.\times\left(B_{i}\left(\frac t i\right)- \sum_{k\ge i+1} \theta(1-\theta)^{k-i-1} B_{i}\left(\frac tk\right)\right)\right)_{t\in[0,1]}.
\end{split}
\end{align}
\end{proposition}

Note that
$$ \E\mathcal{G}_{\theta}^{-1}\ =\ \frac{\theta \log(1/\theta)}{1-\theta}. $$
Three realisations of the process $G$, simulated by using the representation \eqref{eq:g_series}, are shown in the right panel of Figure~\ref{fig:g}.

\begin{figure}[hb]
\includegraphics[width=0.48\textwidth]{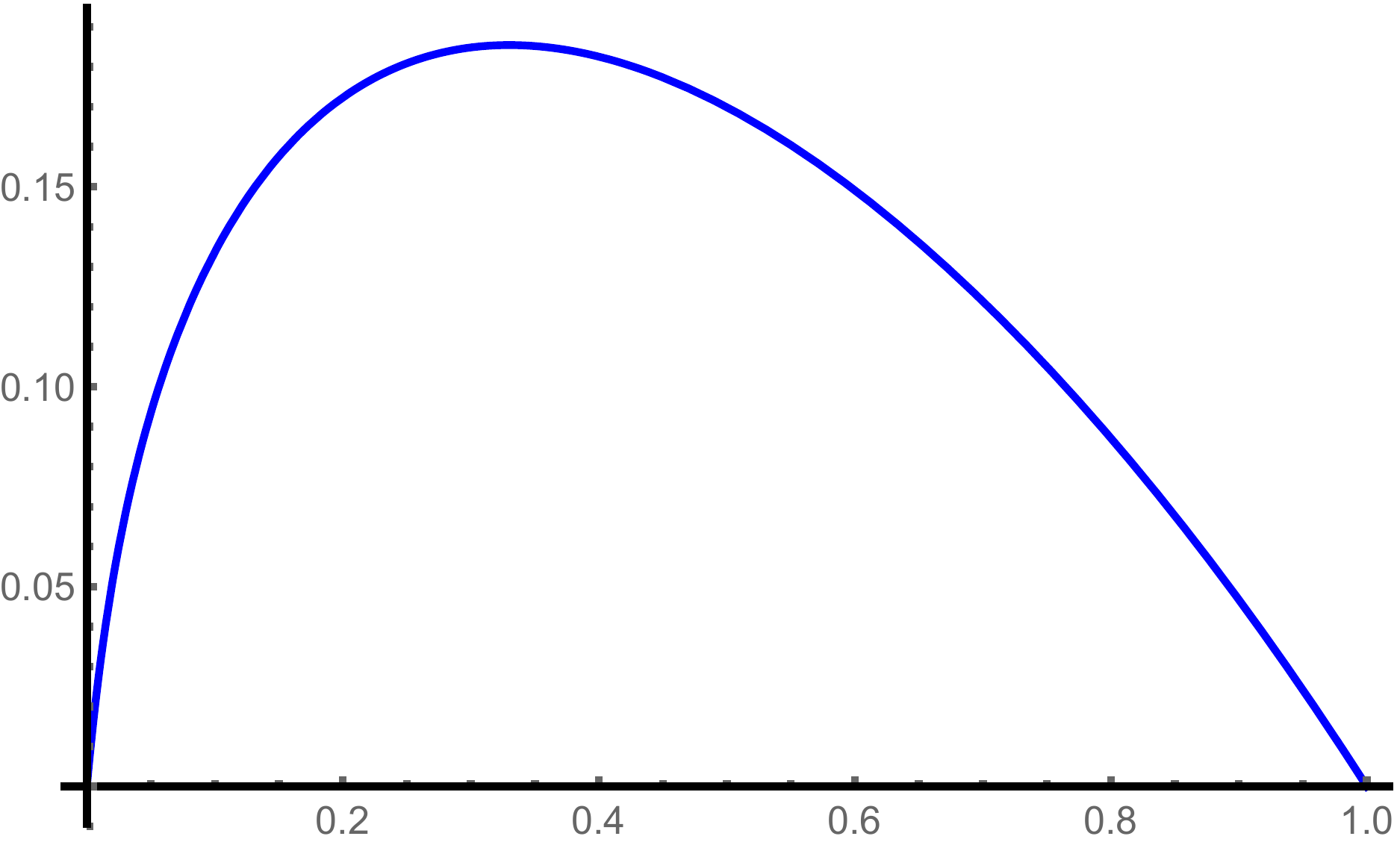}\quad
\includegraphics[width=0.48\textwidth]{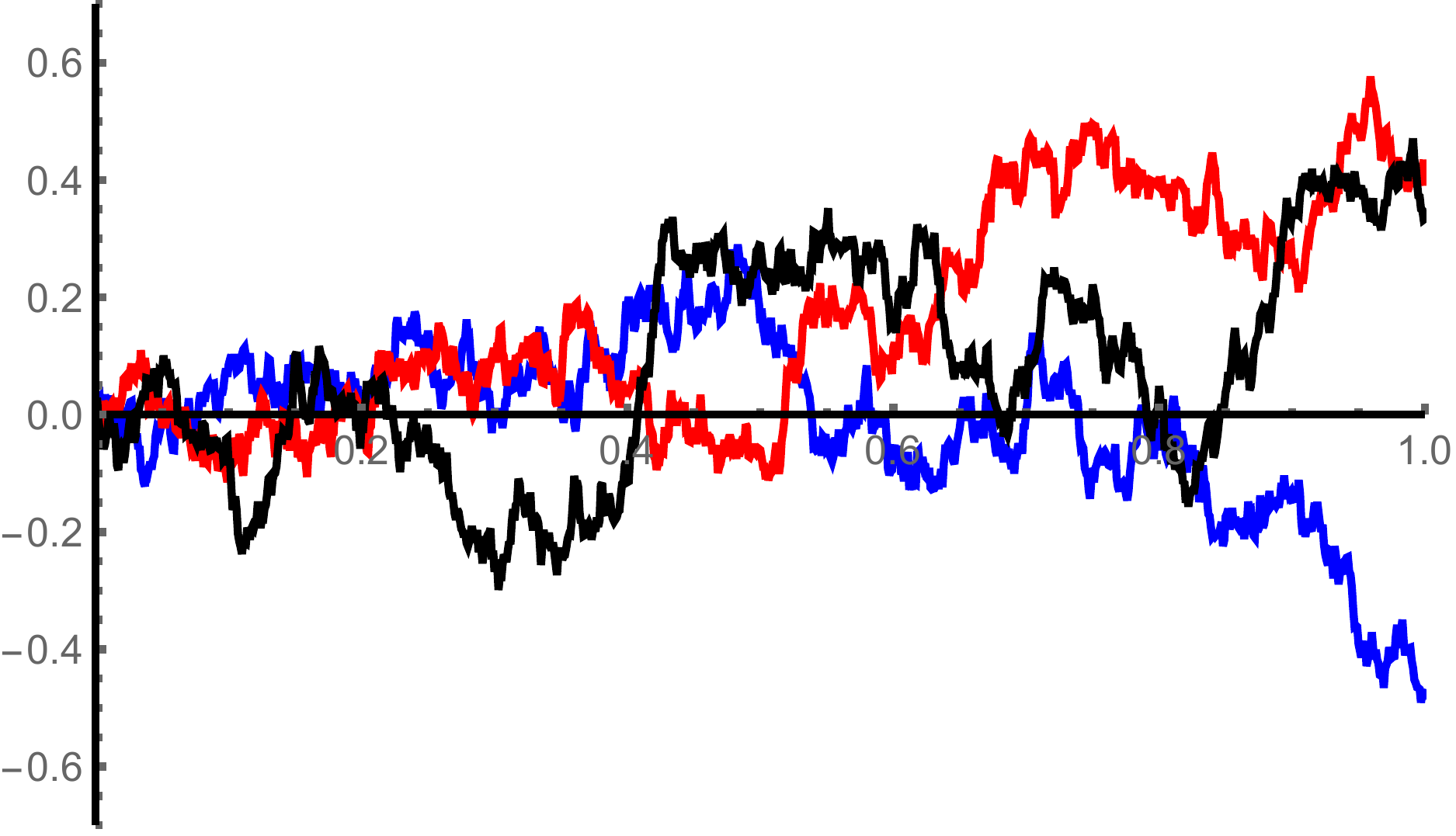}
\caption{The graph of $g(1-\theta)$, $0\le\theta\le 1$ (left) and three realisations of the limit Gaussian process $(G(t))_{t\in[0,1]}$ for $\theta=1/2$ (right).}
\label{fig:g}
\end{figure}

\begin{remark}\label{rem:g in closed form}
As already mentioned, the function $g$ in \eqref{eq:g_def} can be found explicitly,  namely
$$ g(z)\ =\ \frac {z-1}{z^{2}} \big(\log(1-z) + (1+z)\log (1+z)\big), \quad |z|<1. $$

The graph of the variance $g(1-\theta)$, thus the variance of $G(t)/t^{1/2}$ for any $0<t\le 1$, is shown in the left panel of Figure~\ref{fig:g}. Indeed, it follows from \eqref{eq:g_def} that $g(z)=h(z)-h(z^{2})$,
where
$$ h(z)\ =\ \sum_{k\ge 1}\frac{z^{k}}{k(k+1)}\ =\ \sum_{k\ge 1}\frac{z^{k}}{k}-\frac{1}{z}\sum_{k\ge 1}\frac{z^{k+1}}{k+1}\ =\ 1+\frac{1-z}{z}\log (1-z) $$
for $|z|<1$.
\end{remark}

\begin{remark}
It is known, see \cite[Prop.~2.1 and Cor.~2.1]{CillRueSarka:14}, that
\begin{align}
\begin{split}\label{eq:mean}
\E \log L_{n}\ &=\ \theta\sum_{k\ge 1}\psi\left(\frac{n}{k}\right)(1-\theta)^{k}\\
&=\ \frac{n\theta\log (1/\theta)}{1-\theta}\left(1+O\left(\exp\{-C\sqrt{\log (\theta n)}\}\right)\right),
\end{split}
\end{align}
where $\psi$ is the second Chebyshev function, see \eqref{eq:chebyshev_func2_def} below. Assuming the Riemann hypothesis the $O$-term can be substantially improved to
\begin{equation}\label{eq:mean_under_RH}
\E \log L_{n}=\frac{n\theta\log (1/\theta)}{1-\theta}\left(1+O\left(\frac{\log^{2}n}{\sqrt{n}}\right)\right),
\end{equation}
see formula (6.2) in \cite{Schoenfeld:76}.
However, even \eqref{eq:mean_under_RH} does not allow one to replace $\E \log L_{\lfloor nt\rfloor}$ in \eqref{eq:main_convergence} by $nt\theta(1-\theta)^{-1}\log (1/\theta)$.
\end{remark}

The following central limit theorem is an immediate consequence of Theorem \ref{thm:main}.

\begin{corollary}\label{cor:CLT L_n}
As $n\to\infty$,
$$ \frac{\log L_{n}-\E \log L_{n}}{\sqrt{n\log n}}\ \overset{d}{\to}\ \sqrt{g(1-\theta)}\,\mathcal{N}(0,1), $$
where $\mathcal{N}(0,1)$ is a standard normal random variable.
\end{corollary}

Expansion \eqref{eq:mean} for the mean of $\log L_{n}$ in combination with an estimate of its variance provided in \cite{CillRueSarka:14} will also allow us to prove the following strong version of \eqref{eq:WLLN L_n}.

\begin{theorem}\label{thm:SLLN L_n}
As $n\to\infty$,
\begin{equation}\label{eq:SLLN L_n}
\frac{\log L_{n}}{n}\ \toas\ \frac{\theta\log(1/\theta)}{1-\theta}.
\end{equation}
\end{theorem}

\subsection{Poisson limit theorems}
Two further theorems deal with the case when $\theta$ varies with $n$. In the first one it tends to zero at an appropriate speed, namely $\theta=\theta(n)\simeq \frac{\lambda}{n}$ as $n\to\infty$ for some $\lambda>0$. Since the number of points retained in $A_{n}$ is binomial with parameters $n$ and $\theta$ and thus, for large $n$, approximately Poissonian with mean $\lambda$ in the regime just defined, it should not surprise that the limit in the subsequent result is also Poisson. Let $\Pi(\lambda)$ denote a Poisson random variable with mean $\lambda$.

\begin{theorem}\label{thm:main2}
Suppose that, as $n\to\infty$, $\theta=\theta(n)\simeq\frac{\lambda}{n}$ for some $\lambda>0$. Then
$$ \frac{\log L_{n}}{\log n}\ \overset{d}{\to}\ \Pi(\lambda) $$
as $n\to\infty$.
\end{theorem}

Another, in a sense antipodal regime in which the Poisson distribution appears is when $\theta=\theta(n)\to 1$ at an appropriate speed.

\begin{theorem}\label{thm:main3}
Suppose that, as $n\to\infty$, $\theta=\theta(n) = 1- \frac{\lambda+ o(1)}{n}\log n$ for some $\lambda>0$. Then
$$ \frac 1 {\log n} \left(\sum_{k=1}^n \Lambda(k) - \log L_{n}\right)\ \overset{d}{\to}\  \Pi(\lambda/2) $$
as $n\to\infty$, where $\Lambda$ denotes the von Mangoldt function defined by formula \eqref{eq:von_mangoldt_def} below.
\end{theorem}

\section{Preliminaries}
In what follows, we let $\mathcal{P}$ denote the set of prime numbers and $m\N$ the set $\{m,2m,3m,\ldots\}$ of integral multiples of $m\in\N$. Recall that the von Mangoldt function $\Lambda:\N\mapsto \R$ is defined as
\begin{equation}\label{eq:von_mangoldt_def}
\Lambda(n)=
\begin{cases}
\log p, & \text{if } n=p^{k}\text{ for some } k\in\N\text{ and }p\in\mathcal{P},\\
\hfill0, & \text{otherwise}.
\end{cases}
\end{equation}
We will also use the two Chebyshev functions $\vartheta$ and $\psi$. The first Chebyshev function $\vartheta:\R\mapsto\R$ is defined as
\begin{equation}\label{eq:chebyshev_func1_def}
\vartheta(x)\ =\ \sum_{p\in\mathcal{P}:\,p\le x} \log p,
\end{equation}
and the second Chebyshev function $\psi:\R\mapsto \R$ as
\begin{equation}\label{eq:chebyshev_func2_def}
\psi(x)\ =\ \sum_{k\le x}\Lambda(k).
\end{equation}
Recalling the identity
$$ \log\LCM([n])\ =\ \psi(n), $$
we state the following result taken from \cite{CillRueSarka:14}, see Lemma 2.1 therein.

\begin{lemma}\label{lem:lcm_via_mangoldt}
Let $A$ be an arbitrary set of positive integers and $\LCM(A)$ the least common multiple of the elements of $A$. Then
$$ \log \LCM(A)\ =\ \sum_{m}\Lambda(m)I_{A}(m), $$
where
$$ I_{A}(m)\ :=\
\begin{cases}
1, & \text{if } A\cap m\N\ne\varnothing,\\
0, & \text{otherwise}.
\end{cases}
$$
\end{lemma}
\begin{proof}
Since
$$ \log n\ =\ \sum_{p\in\mathcal{P}}\log p \sum_{k\in\N}\1_{\{p^{k} | n\}}\ =\ \sum_{\overset{p\in\mathcal{P},k\in\N}{p^{k} | n}}\log p, $$
we further have
\begin{gather*}
\begin{split}
\log \LCM(A)\ &=\ \sum_{\overset{p\in\mathcal{P},k\in\N}{p^{k} |  \LCM(A)}}\log p\\
&=\ \sum_{p\in\mathcal{P},k\in\N}\log p\,I_{A}(p^{k})\ =\ \sum_{m}\Lambda(m) I_{A}(m),
\end{split}
\shortintertext{where}
p^{k}|\LCM(A)\quad\Longleftrightarrow\quad A\cap\{p^{k},2p^{k},3p^{k},\ldots,\}\ =\ A\cap p^{k}\N\ \ne\ \varnothing
\end{gather*}
should be observed for the second equality.
\end{proof}

\section{Proof of Theorem \ref{thm:main}}

The proof is divided into four steps. The first two steps provide that $\log L_{n}$ is well approximated by a sum of independent random variables. The third step will be to check that finite-dimensional distributions of the approximating sum converge to finite-dimensional distributions of the Gaussian process $(G(t))_{t\in[0,1]}$. In the fourth step, this will be improved to give the asserted functional limit theorem.

\vspace{.2cm}
{\sc Step 1}. By Lemma \ref{lem:lcm_via_mangoldt},
\begin{align*}
\log L_{\lfloor nt\rfloor}\ &=\ \log\LCM(A_{\lfloor nt\rfloor})\ =\ \sum_{m}\Lambda(m)I_{A_{\lfloor nt\rfloor}}(m)\\
&=\ \sum_{p\in\mathcal{P}}\log p\sum_{k\ge 1} I_{A_{\lfloor nt\rfloor}}(p^{k})\\
&=\ \sum_{p\in\mathcal{P}}\log p\, I_{A_{\lfloor nt\rfloor}}(p)\,+\,\sum_{p\in\mathcal{P}}\log p\sum_{k\ge 2} I_{A_{\lfloor nt\rfloor}}(p^{k})\\
&=:\ S_{1}(\lfloor nt\rfloor)\,+\,S_{2}(\lfloor nt\rfloor).
\end{align*}
We will show first that, as $n\to\infty$,
$$ \frac{\sup_{t\in[0,1]}S_{2}(\lfloor nt\rfloor)}{\sqrt{n\log n}}\ \overset{\P}{\to}\ 0 $$
which, by monotonicity of $t\mapsto S_{2}(\lfloor nt\rfloor)$, is equivalent to
$$ \frac{\sum_{p\in\mathcal{P}}\log p\sum_{k\ge 2}I_{A_{n}}(p^{k})}{\sqrt{n\log n}}\ \overset{\P}{\to}\ 0. $$
By Markov's inequality, it suffices to verify
\begin{equation}\label{eq:s2_{n}eg}
\frac{\sum_{p\in\mathcal{P}}\log p\sum_{k\ge 2} \E I_{A_{n}}(p^{k})}{\sqrt{n\log n}}\ \to\ 0
\end{equation}
as $n\to\infty$. To this end, use Boole's inequality to obtain
$$ \E I_{A_{n}}(p^{k})\,=\,\P\{A_{n}\cap p^{k}\N\ne\varnothing\}\,\le\,\left(\sum_{m\le n/p^{k}}\P\{mp^{k}\in A_{n}\}\right)\wedge 1\,\le\,\frac{n\theta}{p^{k}}\wedge 1.
$$
Fix $k\ge 2$ and write
\begin{align*}
\sum_{p\in\mathcal{P}}\log p\,\E I_{A_{n}}(p^{k})\ &=\ \sum_{p\in\mathcal{P}:p^{k}\le n}\log p\,\E I_{A_{n}}(p^{k})\\
&\le\ \sum_{p\in\mathcal{P}:p^{k}\le n}\log p \left(\frac{n\theta}{p^{k}}\wedge 1\right)\\
&=n\theta\sum_{p\in\mathcal{P}:(n\theta)^{1/k}<p\le n^{1/k}}\frac{\log p}{p^{k}}\ +\ \sum_{p\in\mathcal{P}:p\le (n\theta)^{1/k}}\log p\\
&\le n\theta\sum_{p\in\mathcal{P}:p>(n\theta)^{1/k}} \frac{\log p}{p^{k}}\ +\ \sum_{p\in\mathcal{P}:p\le (n\theta)^{1/k}}\log p.
\end{align*}
For the first term in the previous line, Lemma \ref{lem:tail_est} in the Appendix provides the upper bound $Cn^{1/k}$ for all $n\ge 1$ and some $C>0$. For the second sum we use the bound $$ \sum_{p\in\mathcal{P}:p\le (n\theta)^{1/k}}\log p\ \le\ Cn^{1/k}  $$
for all $n$ and some $C$ which follows from $\sum_{p\in\mathcal{P}:p\le x}\log p\simeq x$ as $x\to\infty$, an equivalent form of the prime number theorem. In both esti\-mates, the constant $C$ does not depend on $k$. Summarizing, we arrive at the inequality
\begin{equation}\label{eq:proof_{i}nterm1}
\sum_{p\in\mathcal{P}}\log p\,\E I_{A_{n}}(p^{k})\ \le\ C n^{1/k}
\end{equation}
for all $n\ge 1$ and some positive constant $C$. Returning to \eqref{eq:s2_{n}eg} and noting that summands in the numerator  are nonzero only for $k\le \log_{2} n$, \eqref{eq:proof_{i}nterm1} implies
\begin{align*}
\sum_{p\in\mathcal{P}}\log p\sum_{k\ge 2} \E I_{A_{n}}(p^{k})\ &\le\ C\sum_{k=2}^{\lceil \log_{2} n\rceil }n^{1/k}\\
&\le\ C(\sqrt{n}+n^{1/3}\log_{2} n )=o(\sqrt{n\log n}),
\end{align*}
as $n\to\infty$, and this proves \eqref{eq:s2_{n}eg}.

\vspace{.2cm}
{\sc Step 2.} We start with the decomposition
\begin{gather*}
S_{1}(\lfloor nt\rfloor)\ =\ \sum_{p\in\mathcal{P}}\log p\, I_{A_{\lfloor nt\rfloor}}(p)\ =\ S^{(1,n)}_{1}(t)\ +\ S^{(2,n)}_{1}(t),
\shortintertext{where}
S^{(1,n)}_{1}(t)\ :=\ \sum_{p\in\mathcal{P}:p\le \sqrt{n}}\log p\, I_{A_{\lfloor nt\rfloor}}(p)
\shortintertext{and}
S^{(2,n)}_{1}(t)\ :=\ \sum_{p\in\mathcal{P}:p>\sqrt{n}}\log p\, I_{A_{\lfloor nt\rfloor}}(p).
\end{gather*}
For the first sum, we then proceed as follows. Using the prime number theorem,
$$ \sup_{t\in[0,1]}S^{(1,n)}_{1}(t)\ \le\ \sum_{p\in\mathcal{P}:p\le \sqrt{n}}\log p\ \simeq\  \sqrt{n}
$$
and therefore
\begin{gather*}
\frac{\sup_{t\in[0,1]}S^{(1,n)}_{1}(t)}{\sqrt{n\log n}}\ \overset{\mathbb{P}}{\to}\ 0
\shortintertext{as well as}
\frac{\E S^{(1,n)}_{1}(1)}{\sqrt{n\log n}}\ \to\ 0
\end{gather*}
as $n\to\infty$.

\vspace{.1cm}
In view of what has been shown so far, it remains to prove the asserted limit theorem for $(S^{(2,n)}_{1}(t))_{t\in [0,1]}$, i.e.
$$ \left(\frac{S^{(2,n)}_{1}(t)-\E S^{(2,n)}_{1}(t)}{\sqrt{n\log n}}\right)_{t\in [0,1]}\ \stackrel{J_{1}}{\Longrightarrow}\ (G(t))_{t\in[0,1]}, $$
which is possible because the processes $(I_{A_{\lfloor nt\rfloor}}(p))_{t\in[0,1]}$ and $(I_{A_{\lfloor nt\rfloor}}(q))_{t\in[0,1]}$ are independent for distinct primes $p,q>\sqrt{n}$. For the latter, just observe that the sets $p\N\cap [n]$ and $q\N\cap [n]$ are disjoint for such $p,q$.

\vspace{.2cm}
{\sc Step 3.} Our aim is to show that, as $n\to\infty$,
\begin{equation}\label{eq:proof_fdd_conv}
\left(\frac{\sum_{p\in\mathcal{P}:p>\sqrt{n}}\log p\,\big(I_{A_{\lfloor nt\rfloor}}(p)-\E I_{A_{\lfloor nt\rfloor}}(p)\big)}{\sqrt{n\log n}}\right)_{t\in [0,1]}\ \overset{\textit f.d.d.}{\Longrightarrow}\  (G(t))_{t\in[0,1]}.
\end{equation}
First we show convergence of the covariances. For $0<s\le t\le 1$, we have
\begin{align*}
{\rm Cov}&[S_{1}^{(2,n)}(t),S_{1}^{(2,n)}(s)]\\
&=\sum_{p\in\mathcal{P}\cap(\sqrt{n},nt]}\sum_{q\in\mathcal{P}\cap(\sqrt{n},ns]}\log p\log q\ {\rm Cov}[I_{A_{\lfloor nt\rfloor}}(p),I_{A_{\lfloor ns\rfloor}}(q)]\\
&=\sum_{p\in\mathcal{P}\cap(\sqrt{n},ns]}\log^{2}p\ {\rm Cov}[I_{A_{\lfloor nt\rfloor}}(p),I_{A_{\lfloor ns\rfloor}}(p)]\\
&=\sum_{p\in\mathcal{P}\cap(\sqrt{n},ns]}\log^{2}p \,\E I_{A_{\lfloor ns\rfloor}}(p)\big(1-\E I_{A_{\lfloor nt\rfloor}}(p)\big)\\
&=\sum_{p\in\mathcal{P}\cap(\sqrt{n},ns]}\log^{2}p \left(1-(1-\theta)^{\lfloor ns/p\rfloor}\right)(1-\theta)^{\lfloor nt/p\rfloor},
\end{align*}
where the independence of $(I_{A_{\lfloor nt\rfloor}}(p))_{t\in[0,1]}$ and $(I_{A_{\lfloor nt\rfloor}}(q))_{t\in[0,1]}$ enters when passing to the second equality. By invoking Lemma \ref{lem:cov_asymp} in the Appendix, we infer for $0<s\le t\le 1$
\begin{align*}
\lim_{n\to\infty}&{\rm Cov}\left[\frac{S_{1}^{(2,n)}(t)}{\sqrt{n\log n}},\frac{S_{1}^{(2,n)}(s)}{\sqrt{n\log n}}\right]\\
&=\ \sum_{j\ge 1}(1-(1-\theta)^{j})\hspace{-.5cm}\sum_{\textstyle i\in\left(\frac{tj}{s}-1,\frac{tj}{s}+\frac{t}{s}\right)}\hspace{-.3cm}(1-\theta)^{i} \left(\frac{t}{i}\wedge \frac{s}{j}-\frac{t}{i+1}\vee \frac{s}{j+1}\right)\\
&=:\ C_{1}(t,s).
\end{align*}

In order to prove formula \eqref{eq:G_covar2} for $C_{1}(t,s)$, write the latter in the form
\begin{align*}
C_{1}(t,s)\ &=\ \sum_{i,j\ge 1}(1-(1-\theta)^{j})(1-\theta)^{i}\left(\frac{t}{i}\wedge \frac{s}{j}-\frac{t}{i+1}\vee \frac{s}{j+1}\right)^{+}\\
&=\ \sum_{i,j\ge 1}\left(\sum_{k=1}^{j}p_{k}\right)\left(\sum_{l\ge i+1}p_{l}\right)\left(\frac{t}{i}\wedge \frac{s}{j}-\frac{t}{i+1}\vee \frac{s}{j+1}\right)^{+}\\
&=\ \sum_{k\ge 1}\sum_{l\ge 2}p_{k}p_{l} \sum_{j\ge k}\sum_{i=1}^{l-1}\left(\frac{t}{i}\wedge \frac{s}{j}-\frac{t}{i+1}\vee \frac{s}{j+1}\right)^{+}.
\end{align*}
We claim that the inner double sum equals $(s/k-t/l)^{+}$. Consider two intervals $(t/l,t]$ and $(0,s/k]$. Cover the first interval by the disjoint subintervals $(t/(i+1),t/i]$, $i=1,\ldots,l-1$ and, analogously, the second interval by the disjoint subintervals $(s/(j+1),s/j]$, $j\ge k$. Then $\left(\frac{t}{i}\wedge \frac{s}{j}-\frac{t}{i+1}\vee \frac{s}{j+1}\right)^{+}$ equals the length of the intersection of $(t/(i+1),t/i]$ and $(s/(j+1),s/j]$ and is zero if they are disjoint. The total sum of these lengths equals the length of the intersection of the original intervals $(t/l,t]$ and  $(0,s/k]$, thus $(s/k-t/l)^{+}$. Consequently,
$$ C_{1}(t,s)\ =\ \sum_{k\ge 1}\sum_{l\ge 2}p_{k}p_{l}\left(\frac{s}{k}-\frac{t}{l}\right)^{+}\ =\ \sum_{k\ge 1}\sum_{l\ge 1}p_{k}p_{l}\left(\frac{s}{k}-\frac{t}{l}\right)^{+}, $$
where the second equality holds because $(s/k-t)^{+}=0$. Let $\eta_{1},\eta_{2}$ be two independent geometric random variables on $\N$, viz.
$$ \P\{\eta_{1}=k\}\ =\ \P\{\eta_{2}=k\}\ =\ \theta(1-\theta)^{k-1},\quad k\in\N. $$
Then
\begin{align*}
C_{1}(t,s)\ &=\ \E \left(\frac{s}{\eta_{1}}-\frac{t}{\eta_{2}}\right)^{+}\ =\ \E \left(\frac{s}{\eta_{1}}-\frac{s}{\eta_{1}}\wedge \frac{t}{\eta_{2}}\right)\\
&=\ \sum_{k\geq 1}\frac{s}{k}p_{k}\ -\ \sum_{k,l\ge 1}\left(\frac{s}{k}\wedge \frac{t}{l}\right)p_{k}p_{l}
\end{align*}
which is the asserted result as $s<t$.

\vspace{.1cm}
To complete the proof of \eqref{eq:proof_fdd_conv}, it remains to verify the Lindeberg condition
\begin{equation}\label{eq:proof_{l}indeberg}
\lim_{n\to\infty}\sum_{p\in\mathcal{P}\cap(\sqrt{n},nt]}\E \left[|V_{n,p}(t)|^{2}\1_{\{|V_{n,p}(t)|>\varepsilon \}}\right]\ =\ 0.
\end{equation}
for any $t\in [0,1]$ and $\varepsilon>0$, where
$$ V_{n,p}(t)\ :=\ \frac{\log p\left(I_{A_{\lfloor nt\rfloor}}(p)-\E I_{A_{\lfloor nt\rfloor}}(p)\right)}{\sqrt{n\log n}} $$
for $p\in\mathcal{P}\cap (\sqrt{n},n]$. But this is obvious because $|V_{n,p}(t)|\le \sqrt{\frac{\log n}{n}}$ for all such $p$ and $t\in [0,1]$, $n\in\N$.

\vspace{.2cm}
{\sc Step 4.} In order to finally show the functional limit theorem, we will apply Theorem 10.6 from \cite{Pollard:90} that provides general conditions for the convergence
of triangular arrays of row-wise independent processes to a Gaussian limit. Actually, this theorem yields convergence in the sense of convergence in the space of bounded functions with the usual supremum-norm, which is stronger. Convergence in $D[0,1]$ with the $J_{1}$-topology follows as a direct consequence.

To conform with the notation in  \cite{Pollard:90}, put
$$ f_{n,p}(t):=\frac{\log p\, I_{A_{\lfloor nt\rfloor}}(p)}{\sqrt{n\log n}}\quad\text{and}\quad F_{n,p}:=\frac{\log p\, I_{A_{n}}(p)}{\sqrt{n\log n}} $$
for $p\in\mathcal{P}\cap (\sqrt{n},n]$. Conditions (ii) and (iv) of Theorem 10.6 in \cite{Pollard:90} were checked in Step 3. Condition (iii) is obvious. The manageability of the family $(f_{n,p}(\cdot))_{p}$ (Condition (i) of Theorem 10.6 in \cite{Pollard:90}) follows from the monotonicity of
$(f_{n,p}(t))_{t\in [0,1]}$ in $t$ for every fixed $n$ and $p$, and the observation in the paragraph just before Theorem 11.17 in \cite[p.~221]{Kosorok:08}. It remains to verify condition (v). To this end, introduce the function
$$ \rho_{n}(s,t)\ :=\ \left(\frac{\sum_{p\in\mathcal{P}\cap (\sqrt{n},n]}\log^{2}p\ \E | I_{A_{\lfloor nt\rfloor}}(p)-I_{A_{\lfloor ns\rfloor}}(p)|^{2}}{n\log n}\right)^{1/2} $$
for $0\le s, t\le 1$. Note that
\begin{align*}
&\E | I_{A_{\lfloor nt\rfloor}}(p)-I_{A_{\lfloor ns\rfloor}}(p)|^{2}\ =\ \P\{I_{A_{\lfloor nt\rfloor}}(p)-I_{A_{\lfloor ns\rfloor}}(p)=1\}\\
&\hspace{4.11cm}=\ 1-\P\{I_{A_{\lfloor nt\rfloor}}(p)-I_{A_{\lfloor ns\rfloor}}(p)=0\}\\
&\hspace{4.11cm}=\ (1-\theta)^{\lfloor ns/p\rfloor}-(1-\theta)^{\lfloor nt/p\rfloor}.
\shortintertext{and therefore}
&\rho_{n}(s,t) := \left(\frac{\sum_{p\in\mathcal{P}\cap (\sqrt{n},n]}\log^{2}p \left((1-\theta)^{\lfloor ns/p\rfloor}-(1-\theta)^{\lfloor nt/p\rfloor}\right)}{n\log n}\right)^{1/2}
\end{align*}
for $0<s\le t\le 1$. Decomposing the numerator on the right-hand side as
\begin{align*}
\sum_{p\in\mathcal{P}\cap (\sqrt{n},n]}&\log^{2}p\, \left((1-\theta)^{\lfloor ns/p\rfloor}-(1-\theta)^{\lfloor nt/p\rfloor}\right)\\
&=\ \sum_{p\in\mathcal{P}\cap (\sqrt{n},ns]}\log^{2}p\, (1-\theta)^{\lfloor ns/p\rfloor}\ +\ \sum_{p\in\mathcal{P}\cap (ns,n]}\log^{2}p\\
&\hspace{.1cm}-\ \sum_{p\in\mathcal{P}\cap (\sqrt{n},nt]}\log^{2}p\, (1-\theta)^{\lfloor nt/p\rfloor}\ -\  \sum_{p\in\mathcal{P}\cap (nt,n]}\log^{2}p,
\end{align*}
and applying Lemma \ref{lem:simple_sum_asymp} in conjunction with formula \eqref{eq:log^{2}_asymp} in the Appendix, we deduce
$$ \lim_{n\to\infty}\rho_{n}(s,t)=\sqrt{(1-h(1-\theta))(t-s)} $$
for $0<s\le t\le 1$. Now let $(s_{n})_{n\ge 1}$ and $(t_{n})_{n\ge 1}$ be two deterministic sequences in $[0,1]$ such that $s_{n}-t_{n}\to 0$ as $n\to\infty$. We must show that
\begin{gather*}
\lim_{n\to\infty}\rho_{n}(s_{n},t_{n})=0,
\shortintertext{or, equivalently,}
\lim_{n\to\infty}\frac{\sum_{p\in\mathcal{P}\cap (\sqrt{n},n]}\log^{2}p \left|(1-x)^{\lfloor ns_{n}/p\rfloor}-(1-x)^{\lfloor nt_{n}/p\rfloor}\right|}{n\log n }\ =\ 0.
\end{gather*}
Putting $u_{n}:=s_{n}\wedge t_{n}$ and $v_{n}:=s_{n}\vee t_{n}$, this follows if
\begin{equation}\label{eq:pollards_cond5}
\lim_{n\to\infty}\frac{\sum_{p\in\mathcal{P}:p\le n}\log p \left((1-\theta)^{\lfloor nu_{n}/p\rfloor}-(1-\theta)^{\lfloor nv_{n}/p\rfloor}\right)}{n}\ =\ 0.
\end{equation}
Using Lemma \ref{lem:cheb_sum_with_remainder}, we find that, for a suitable constant $C>0$,
\begin{align*}
\frac{1}{n}\sum_{p\in\mathcal{P}:p\le n}&\log p\left((1-\theta)^{\lfloor nu_{n}/p\rfloor}-(1-\theta)^{\lfloor nv_{n}/p\rfloor}\right)\\
&\le\ (v_{n}-u_{n})h(1-\theta)\ +\ \frac{Cu_{n}}{\log (nu_{n}+2)}\ +\ \frac{Cv_{n}}{\log (nv_{n}+2)}\\
&\le\ (v_{n}-u_{n})h(1-\theta)\ +\ \frac{2C}{\log (n+2)},
\end{align*}
and the last line converges to $0$ because $v_{n}-u_{n}=|s_{n}-t_{n}|\to 0$, as $n\to\infty$.

It remains to note that Theorem 10.6 in \cite{Pollard:90} guarantees that the limit process a.s. has uniformly continuous paths. This completes the proof of Theorem \ref{thm:main}.\qed

\section{Proof of Proposition \ref{prop:G_process}}
(a) Since $(G(t))_{t\in[0,1]}$ defined in Theorem \ref{thm:main} is Gaussian, the same holds true for the process $(\E[B(t/\mathcal{G}_\theta)|B])_{t\in[0,1]}$ as one can readily see from the representation
\begin{align*}
\E[B(t/\mathcal{G}_\theta)|B]\ &=\ \sum_{k\ge 1}p_{k} B(t/k)\ =\sum_{k\ge 1}p_{k} \int_{0}^{\infty}\1_{\{z\le t/k\}}\ {\rm d}B(z)\\
&=\ \int_{0}^{\infty}\left(\sum_{k\ge 1}p_{k}\1_{\{z\le t/k\}}\right)\ {\rm d}B(z).
\end{align*}
By the independence assumption, it is enough to check the equality of covariances which follows immediately from the identities
\begin{align*}
{\rm Cov}[\E[B(t/\mathcal{G}_\theta)|B],\E[B(s/\mathcal{G}_\theta)|B]]\ &=\
{\rm Cov}\left[\sum_{k\ge 1}p_{k} B(t/k),\sum_{l\ge 1}p_{l}B(s/l)\right]\\
&=\ \sum_{k,l\ge 1}p_{k}p_{l}\left(\frac{t}{k}\wedge \frac{s}{l}\right)
\end{align*}
and
$$ {\rm Cov}[B(t(\E \mathcal{G}_{\theta}^{-1})),B(s(\E \mathcal{G}_{\theta}^{-1}))]\ =\ (t\wedge s)\E \mathcal{G}_{\theta}^{-1}\ =\ \sum_{k\ge 1}\left(\frac{s}{k}\wedge \frac{t}{k}\right)p_{k} $$
for $0\le s,t\le 1$.

\vspace{.2cm}
(b) Since the series on the right-hand side of \eqref{eq:g_series} is a centered Gaussian process, it suffices again to check the equality of covariances. We have
\begin{gather*}
\sqrt{\theta(1-\theta)}\sum_{i\ge 1} (1-\theta)^{(i-1)/2}\left(B_{i}\left(\frac t i\right)- \sum_{k\ge i+1} \theta(1-\theta)^{k-i-1} B_{i}\left(\frac tk\right)\right)\\
=:\ \sum_{i\ge 1}\sum_{k\ge i}a_{ik}B_{i}\left(\frac{t}{k}\right),
\end{gather*}
where
$$
a_{ik}:=
\begin{cases}
\hfill\theta^{1/2} (1-\theta)^{i/2}, & \text{if }k=i,\\
-\theta^{3/2}(1-\theta)^{k-i/2-1}, & \text{if }k>i.\\
\end{cases}
$$
Using independence of $B_{1},B_2,\ldots$ and the formula ${\rm Cov}[B_{i}(s),B_{i}(t)]=s\wedge t$, we obtain for its covariance function
\begin{align*}
\rho(s,t)\ :=~&{\rm Cov}\left[\sum_{i\ge 1}\sum_{k\ge i}a_{ik}B_{i}\left(\frac{t}{k}\right),\sum_{j\ge 1}\sum_{l\ge j}a_{jl}B_j\left(\frac{s}{l}\right)\right]\\
=~&\sum_{i\ge 1}\sum_{k\ge i}\sum_{l\ge i}a_{ik}a_{il}\left(\frac{t}{k}\wedge \frac{s}{l}\right)\ =\ \sum_{k\ge 1}\sum_{l\ge 1}\left(\frac{t}{k}\wedge \frac{s}{l}\right)\sum_{i=1}^{k\wedge l}a_{ik}a_{il}.
\end{align*}
If $k=l$, then
\begin{align*}
\sum_{i=1}^{k\wedge l}a_{ik}a_{il}\ &=\ \sum_{i=1}^{k}a_{ik}^{2}\ =\ \theta(1-\theta)^{k}+\theta^{3}\sum_{i=1}^{k-1} (1-\theta)^{2k-i-2}\\
&=\ \theta(1-\theta)^{k}+\theta^{2}(1-\theta)^{k-1}(1-(1-\theta)^{k-1})\\
&=\ p_{k}-p_{k}^{2}\ =\ p_{k}-p_{k}p_{l},
\shortintertext{while for $k<l$}
\sum_{i=1}^{k\wedge l}a_{ik}a_{il}\ &=\ a_{kk}a_{kl}+\sum_{i=1}^{k-1}a_{ik}a_{il}\\
&=\ -\theta^{2}(1-\theta)^{l-1}
+
\theta^{3}\sum_{i=1}^{k-1}(1-\theta)^{k+l-i-2}\ =\ -p_{k}p_{l}.
\end{align*}
A combination of these results yields
\begin{align*}
\rho(s,t)\ =\ \sum_{k\ge 1}\frac{s\wedge t}{k}p_{k}\,-\,\sum_{k,l\ge 1}\left(\frac{t}{k}\wedge \frac{s}{l}\right)p_{k}p_{l},
\end{align*}
which shows the desired equality of distributions and completes the proof of Proposition \ref{prop:G_process}.\qed

\section{Proof of Theorem \ref{thm:SLLN L_n}}\label{sec:SLLN L_n}

For $\varepsilon>0$ and $n\in\N$, we define the events
$$ A_{n}(\varepsilon)\ :=\ \left\{|\log L_{n}-\E\log L_{n}|>\varepsilon\,\E\log L_{n}\right\}. $$
Proposition 2.2 from \cite{CillRueSarka:14} provides us with
\begin{equation}\label{eq:variance}
{\rm Var}[\log L_{n}]\ =\ O(n\log^{2}n)
\end{equation}
as $n\to\infty$ which in combination with the expansion of $\E\log L_{n}$ in \eqref{eq:mean} implies that
\begin{equation}\label{eq:bound A_n(eps)}
\P\{A_{n}(\varepsilon)\}\ \le\ \frac{C\log^{2}(n+1)}{n}
\end{equation}
for all $n\ge 1$ and some constant $C=C(\varepsilon,\theta)>0$. Putting $n_{k}:=[k\log^{4}k]$ for $k\ge 1$, it follows from \eqref{eq:bound A_n(eps)} that $\sum_{k\ge 2}\P\{A_{n_{k}}(\varepsilon)\}<\infty$ for any $\varepsilon>0$ and thus
\begin{equation}\label{eq:SLLN L_n_k}
\lim_{k\to\infty}\frac{\log L_{n_{k}}}{\E\log L_{n_{k}}}\ =\ 1\quad\text{a.s.}
\end{equation}
by the Borel-Cantelli lemma. For arbitrary $n\in\N$, let $k(n)$ be such that $n_{k(n)}\le n<n_{k(n)+1}$ and notice that, as a trivial consequence of \eqref{eq:mean},
$$ \lim_{k\to\infty}\frac{\E\log L_{n_{k+1}}}{\E\log L_{n_{k}}}\ =\ 1. $$
The proof of the theorem is now completed by a combination of the latter fact with \eqref{eq:SLLN L_n_k} and the inequalities
$$ \frac{\log L_{n_{k}}}{\E\log L_{n_{k}}}\cdot\frac{\E\log L_{n_{k}}}{\E\log L_{n_{k+1}}}\ \le\ \frac{\log L_{n}}{\E\log L_{n}}\ \le\ \frac{\log L_{n_{k+1}}}{\E\log L_{n_{k+1}}}\cdot\frac{\E\log L_{n_{k+1}}}{\E\log L_{n_{k}}}, $$
valid for any $n\ge 1$.\qed

\section{Proof of Theorems \ref{thm:main2} and \ref{thm:main3}}

\subsection{Proof of Theorem \ref{thm:main2}}
Let $R(n)$ be the number of remaining integers in $[n]$, i.e. $R(n):=|A_{n}|$, and note that $R(n)$ has a binomial distribution with parameters $n$ and $\theta(n)$. Since $\theta(n)\simeq\lambda/n$, the classical Poisson limit theorem gives
$$ R(n)\ \overset{d}{\to}\ \Pi(\lambda) $$
as $n\to\infty$. Let $(X_{1}^{(n)},X_{2}^{(n)},\ldots,X_{R(n)}^{(n)})$ denote the ordered sample of remaining integers which, conditioned upon $R(n)=k$, has the same distribution as an ordered $k$-sample without replacement from the set $[n]$. In order to show that, as $n\to\infty$,
$$ \frac{\log \LCM(X_{1}^{(n)},X_{2}^{(n)},\ldots,X_{R(n)}^{(n)})}{\log n}\ \overset{d}{\to}\  \Pi(\lambda), $$
it is enough to show that, conditioned upon $R(n)=k$ for any fixed $k\in\N$,
\begin{equation}\label{eq:poisson_limit_goal1}
\frac{\log \LCM(X_{1}^{(n)},X_{2}^{(n)},\ldots,X_{k}^{(n)})}{\log n}\ \overset{\P}{\to}\ k 
\end{equation}
as $n\to\infty$. Given any finite set of positive integers $\{n_{1},n_{2},\ldots,n_{k}\}$, we have that
$$
\frac{n_{1}n_{2}\cdots n_{k}}{\prod_{1\le i<j\le k}\GCD(n_{i},n_j)}\ \le\ \LCM(n_{1},n_{2},\ldots,n_{k})\ \le\ n_{1}n_{2}\cdots n_{k}.
$$
For \eqref{eq:poisson_limit_goal1}, it hence suffices to verify that, given $R(n)=k$,
\begin{gather}
\frac{\sum_{i=1}^{k}\log X_{i}^{(n)}}{\log n}\ \overset{\P}{\to}\ k
\label{eq:poisson_limit_goal2}
\shortintertext{and}
\frac{\sum_{1\leq i < j \leq k}\log \GCD(X_{i}^{(n)},X_{j}^{(n)})}{\log n}\ \overset{\P}{\to}\ 0.
\label{eq:poisson_limit_goal3}
\end{gather}
Let $(U_{k}^{(n)})_{k\in\N}$ be a sequence of i.i.d. random variables with a uniform distribution on $[n]$. Then, as already stated above, the conditional law of $(X_{1}^{(n)},\ldots,X_{k}^{(n)})$ given $R(n)=k$ for any fixed $k\in\N$, is the same as the conditional law of $(U_{(1)}^{(n)},\ldots,U_{(k)}^{(n)})$, the order statistics of $(U_{1}^{(n)},\ldots,U_{k}^{(n)})$, given the event
$$ A_{n,k}\ :=\ \left\{U_{i}^{(n)}\ne U_j^{(n)}:\,i,j=1,\ldots, k,i\ne j,\right\}. $$
Since $\P\{A_{n,k}\}$ tends to 1 for any $k\in\N$ and $n\to\infty$, \eqref{eq:poisson_limit_goal2} and \eqref{eq:poisson_limit_goal3} are equivalent to
\begin{gather}
\frac{\sum_{i=1}^{k}\log U_{i}^{(n)}}{\log n}\ \overset{\P}{\to}\ k\label{eq:poisson_limit_goal4}
\shortintertext{and}
\frac{\log \GCD(U_{1}^{(n)},U_{2}^{(n)})}{\log n}\overset{\P}{\to} 0,\label{eq:poisson_limit_goal5}
\end{gather}
respectively. Assertion \eqref{eq:poisson_limit_goal4} follows directly from
$$ \lim_{n\to\infty}\P\left\{1-\varepsilon<\frac{\log U_{1}^{(n)}}{\log n}\le 1 \right\}\ =\ \lim_{n\to\infty}\P\{n^{1-\varepsilon}<U_{1}^{(n)}\le n\}\ =\ 1 $$
for any $\varepsilon\in (0,1)$ and Slutsky's lemma.

\vspace{.1cm}
For \eqref{eq:poisson_limit_goal5}, we will in fact prove the stronger result that, as $n\to\infty$,
\begin{equation}\label{eq:gcd_conv}
\log \GCD(U_{1}^{(n)},U_{2}^{(n)})\ \overset{d}{\to}\ \xi
\end{equation}
for some proper nondegenerate random variable $\xi$ to be defined below. Writing
$$
U_{1}^{(n)}\ =\ \prod_{p\in\mathcal{P}}p^{\lambda_{p}(U_{1}^{(n)})}\quad\text{and}\quad U_{2}^{(n)}\ =\ \prod_{p\in\mathcal{P}}p^{\lambda_{p}(U_{2}^{(n)})},
$$
where $\lambda_{p}(m)\ge 0$ is the power of prime $p$ in the prime decomposition of $m\in\N$, we have
$$ \log \GCD(U_{1}^{(n)},U_{2}^{(n)})\ =\ \sum_{p\in\mathcal{P}}\left(\lambda_{p}(U_{1}^{(n)})\wedge \lambda_{p}(U_{2}^{(n)})\right)\log p. $$
It is a simple fact, see for example the last display on p.~28 in \cite{ArrBarTav:03}, that
\begin{equation}\label{eq:conv_to_geom}
\left(\lambda_{p}(U_{1}^{(n)})\right)_{p\in\mathcal{P}}\ \overset{d}{\to}\ \left(Z_{1,p}\right)_{p\in\mathcal{P}},
\end{equation}
where $(Z_{1,p})_{p\in\mathcal{P}}$ forms a sequence of independent random variables and $Z_{1,p}$ has a geometric distribution on $\{0,1,2,\ldots,\}$ with parameter $1-\frac{1}{p}$. Likewise,
\begin{equation*}
\left(\lambda_{p}(U_{2}^{(n)})\right)_{p\in\mathcal{P}}\ \overset{d}{\to}\ \left(Z_{2,p}\right)_{p\in\mathcal{P}},
\end{equation*}
where $(Z_{1,p})_{p\in\mathcal{P}}\overset{d}{=}(Z_{2,p})_{p\in\mathcal{P}}$, $(Z_{1,p})_{p\in\mathcal{P}}$ and  $(Z_{2,p})_{p\in\mathcal{P}}$ are independent. The series
$$ \xi\ :=\ \sum_{p\in\mathcal{P}}(Z_{1,p}\wedge Z_{2,p})\log p $$
converges a.s. because it has finite mean, viz.
\begin{equation}\label{eq:series_converges}
\sum_{p\in\mathcal{P}}\E (Z_{1,p}\wedge Z_{2,p})\log p\ =\ \sum_{p\in\mathcal{P}}\frac{\log p}{p^{2}-1}<\infty.
\end{equation}
We note in passing that the explicit form of the distribution of $\xi$ may be found in \cite{DiaconisErdos:04}.
According to Theorem 3.2 in \cite{Billingsley:99}, a sufficient condition for \eqref{eq:gcd_conv} is that
\begin{equation}\label{eq:billingsley_cond}
\lim_{m\to\infty}\limsup_{n\to\infty}\P\left\{\sum_{p\in\mathcal{P}:p\ge m}\left(\lambda_p(U_{1}^{(n)})\wedge\lambda_p(U_{2}^{(n)})\right)\log p\ge\varepsilon \right\}\ =\ 0
\end{equation}
for any $\varepsilon>0$. We will show the in fact stronger condition (by Markov's inequality)
\begin{equation}\label{eq:billingsley_cond2}
\lim_{m\to\infty}\limsup_{n\to\infty}\sum_{p\in\mathcal{P}:p\ge m}\E\left(\lambda_p(U_{1}^{(n)})\wedge\lambda_p(U_{2}^{(n)})\right)\log p\ =\ 0.
\end{equation}
To this end, note that
\begin{gather*}
\E \left(\lambda_p(U_{1}^{(n)})\wedge\lambda_p(U_{2}^{(n)})\right)\ =\ \sum_{i\ge 1}\P\{\lambda_p(U_{1}^{(n)})\ge i,\lambda_p(U_{2}^{(n)})\ge i\}\\
=\ \sum_{i\ge 1}\left(\frac{1}{n}\left\lfloor\frac{n}{p^{i}}\right\rfloor\right)^{2}\ \le\ \sum_{i\ge 1}\frac{1}{p^{2i}}\ =\ \frac{1}{p^{2}-1}.
\end{gather*}
Relation \eqref{eq:billingsley_cond2} now follows from \eqref{eq:series_converges}, thus completing the proof of Theorem \ref{thm:main2}.\qed

\subsection{Proof of Theorem \ref{thm:main3}}
Using Lemma \ref{lem:lcm_via_mangoldt}, we can write
$$ \sum_{k=1}^{n}\Lambda(k)-\log L_{n}\ =\ \sum_{k=1}^{n}\Lambda(k)\left(1-I_{A_{n}}(k)\right) $$
and infer
$$ \sum_{k\le n/2}\Lambda(k)\left(1-I_{A_{n}}(k)\right)\ \overset{\P}{\to}\ 0\quad (n\to\infty) $$
from 
\begin{gather*}
\P\left\{\sum_{k\le n/2}\Lambda(k)\left(1-I_{A_{n}}(k)\right)>0\right\}\ \leq \ \P\{I_{A_{n}}(k)=0\text{ for some }k\le n/2\}\\
\le\ \sum_{k\le n/2}\P\{I_{A_{n}}(k)=0\}\ \le\ \sum_{k\le n/2}\P\{k\notin A_{n},2k\notin A_{n}\}\\
\le \frac{n}{2}\,\frac{(\lambda+o(1))^{2}\log^{2}n}{n^{2}}.
\end{gather*}

Left with the sum $\sum_{n/2<k\le n}\Lambda(k)\left(1-I_{A_{n}}(k)\right)$, we note that the random variables $\{1-I_{A_{n}}(k): n/2<k\le n\}$ are independent indicators satisfying
$$ \P\{1-I_{A_{n}}(k)=1\}\ =\ \P\{k\notin A_{n}\}\ =\ 1-\theta(n)\ \simeq\  \frac{\lambda \log n}{n} $$
as $n\to\infty$. By definition of the von Mangoldt function $\Lambda$, we have
\begin{align*}
\sum_{n/2<k\le n}\Lambda(k)\left(1-I_{A_{n}}(k)\right)\ &=\ \sum_{p\in\mathcal{P}}\log p \left(1-I_{A_{n}}(p)\right)\1_{\{n/2<p\le n\}}\\
%&=\sum_{p\in\mathcal{P}}\sum_{l\ge 1}\log p \left(1-I_{A_{n}}(p^l)\right)\1_{\{n/2<p^l \le n\}}\\
&+\sum_{p\in\mathcal{P}}\sum_{l\ge 2}\log p \left(1-I_{A_{n}}(p^l)\right)\1_{\{n/2<p^l\le n\}}.
\end{align*}
The expectation of the last term on the right-hand side equals
\begin{align*}
(1-\theta(n))&\sum_{p\in\mathcal{P}}\sum_{l\ge 2}\log p \1_{\{n/2<p^l\le n\}}\\
&=\ \frac{(\lambda+o(1)) \log n}{n}\sum_{p\in\mathcal{P}}\sum_{l\ge 2}\log p\,\1_{\{n/2<p^l\le n\}}\\
&\le\ \frac{(\lambda+o(1)) \log n}{n}\sum_{p\in\mathcal{P}}\sum_{l\ge 2}\log p\, \1_{\{p^l\le n\}}\\
&=\ \frac{(\lambda+o(1))\log n}{n}(\psi(n)-\vartheta(n)),
\end{align*}
where $\psi$ and $\vartheta$ are the Chebyshev functions, see Formulae \eqref{eq:chebyshev_func1_def} and \eqref{eq:chebyshev_func2_def}. By Theorem 4.1 in \cite{Apostol:76},
$$ \psi(n)-\vartheta(n)\ =\ O(\sqrt{n}\log^{2}n) $$
as $n\to\infty$, and this in combination with Markov's inequality implies
$$ \sum_{p\in\mathcal{P}}\sum_{l\ge 2}\log p \left(1-I_{A_{n}}(p^l)\right)\1_{\{n/2<p^l\le n\}}\ \overset{\P}{\to}\ 0 $$
as $n\to\infty$. It remains to show that
$$
\frac{\sum_{p\in\mathcal{P}}\log p \left(1-I_{A_{n}}(p)\right)\1_{\{n/2<p\le n\}}}{\log n}\ \overset{d}{\to}\ \Pi(\lambda/2).
$$
In view of the obvious inequalities
\begin{align*}
\log(n/2)\sum_{p\in\mathcal{P}}\left(1-I_{A_{n}}(p)\right)\1_{\{n/2<p\le n\}}\ &\le\  \sum_{p\in\mathcal{P}}\log p \left(1-I_{A_{n}}(p)\right)\1_{\{n/2<p\le n\}}\\
&\le\ \log n\sum_{p\in\mathcal{P}}\left(1-I_{A_{n}}(p)\right)\1_{\{n/2<p\le n\}},
\end{align*}
the claim of the theorem follows from the observation that
$$ \sum_{p\in\mathcal{P}\cap (n/2,n]}\left(1-I_{A_{n}}(p)\right)\ =\ \sum_{p\in\mathcal{P}\cap (n/2,n]}\1_{\{p\notin A_{n}\}}\ \overset{d}{\to}\  \Pi(\lambda/2),\quad n\to\infty.$$
The latter convergence holds by the classic Poisson limit theorem for independent indicators, here Bernoulli variables with parameter $1-\theta(n)$. The factor $1/2$ in the parameter of the Poisson random variable appears because, with $\pi(x)$ denoting the number of primes $\le x$,  the number of summands is $\pi(n)-\pi(n/2)\sim \frac{n}{2\log n}$, as $n\to\infty$ by the prime number theorem. The proof of Theorem \ref{thm:main3} is complete.\qed

\section{Appendix}
We have used the following estimate for the tails of convergent series involving primes.
\begin{lemma}\label{lem:tail_est}
There exists a positive constant $C$ such that, for all $n\in\N$ and $k\ge 2$,
\begin{equation}\label{eq:tail_est}
\sum_{p\in\mathcal{P}:p\ge n}\frac{\log p}{p^{k}}\ \le\ \frac{C}{n^{k-1}}.
\end{equation}
\end{lemma}
\begin{proof}
For any $n\ge 2$ and with $\pi(x)$ as above, integration by parts yields
\begin{align*}
\sum_{p\in\mathcal{P}:p\ge n}\frac{\log p}{p^{k}}\ &=\ \int_{[n,\infty)}\frac{\log x}{x^{k}}\ {\rm d}\pi(x)\\
&=\ \frac{\log x}{x^{k}}\pi(x)\bigg|_{n}^{\infty}\ +\ \int_{n}^{\infty}\pi(x)\frac{kx^{k-1}\log x-x^{k-1}}{x^{2k}}\ {\rm d}x\\
&\le\ \int_{n}^{\infty}\pi(x)\frac{kx^{k-1}\log x}{x^{2k}}\ {\rm d}x.
\end{align*}
By the prime number theorem, $\pi(x)\le C_{1}x/\log x$ for some constant $C_{1}>0$ and all $x\ge 2$. Consequently,
$$ \sum_{p\in\mathcal{P}:p\ge n}\frac{\log p}{p^{k}}\ \le\ C_{1}k\int_{n}^{\infty}\frac{{\rm d}x}{x^{k}}\ =\ \frac{C_{1}k}{k-1}n^{1-k}\le \frac{2C_{1}}{n^{k-1}} $$
and thus \eqref{eq:tail_est} holds with $C:=2C_{1}$.
\end{proof}

\begin{lemma}\label{lem:simple_sum_asymp}
For any fixed $x,t\in (0,1)$,
$$ \sum_{p\in\mathcal{P}:p\in (\sqrt{n},nt]}\log^{2}p\,(1-x)^{\lfloor nt/p\rfloor}\ \simeq\ t\cdot n\log n\cdot h(1-x), $$
where $h(x)=\sum_{k\ge 1}\frac{x^{k}}{k(k+1)}$.
\end{lemma}
\begin{proof}
Let us first show that
\begin{equation}\label{eq:log^{2}_asymp}
\sum_{p\in\mathcal{P}:p\le x}\log^{2}p\ \simeq\ x\log x.
\end{equation}
as $x\to\infty$. To this end, note that
\begin{align*}
\sum_{p\le x,p\in\mathcal{P}}\log^{2}p\ =\ \int_{[2,x]} \log^{2}z\ {\rm d}\pi(z)\ =\ \pi(z)\log^{2}z\bigg|_{2}^{x}-\int_{2}^x \frac{2\pi(z)\log z}{z}\ {\rm d}z.
\end{align*}
By another appeal to the prime number theorem, $\pi(x)\log^{2}x\simeq x\log x$ and the integrand is bounded, i.e., the integral itself is $O(x)$ as $x\to\infty$. This proves \eqref{eq:log^{2}_asymp} which in turn further provides us with the relation
\begin{equation}\label{eq:log^{2}_asymp2}
\sum_{p\in\mathcal{P}\cap\,(nt/(k+1),nt/k]}\log^{2}p\ \simeq\ \frac{nt\log n}{k(k+1)}\quad (n\to\infty)
\end{equation}
for any $k\in\N$ and $t\in(0,1]$.

Now fix an arbitrary $m\in\N$ and write
\begin{align*}
\sum_{p\in\mathcal{P}\cap (\sqrt{n},nt]}(1-x)^{\lfloor nt/p\rfloor}\log^{2}p\ &=\ \sum_{k=1}^{\lfloor\sqrt{n}t\rfloor}(1-x)^{k}\sum_{p\in\mathcal{P}\cap\,(nt/(k+1),nt/k]}\log^{2}p\\
&=\ \sum_{k=1}^{m}(1-x)^{k}\sum_{p\in\mathcal{P}\cap\,(nt/(k+1),nt/k]}\log^{2}p\\
&\qquad+\sum_{k=m+1}^{\lfloor\sqrt{n}t\rfloor}(1-x)^{k}\sum_{p\in\mathcal{P}\cap\,(nt/(k+1),nt/k]}\log^{2}p\\
&=:\ A_{1}(n,m)+A_{2}(n,m).
\end{align*}
By \eqref{eq:log^{2}_asymp2}, we have
$$ \lim_{n\to\infty}\frac{A_{1}(n,m)}{n\log n}\ =\ t\sum_{k=1}^{m}\frac{(1-x)^{k}}{k(k+1)}, $$
and for $A_{2}(n,m)$, the estimate
\begin{align*}
A_{2}(n,m)\ &\le\ \sum_{k\ge m+1}(1-x)^{k}\sum_{p\in\mathcal{P}:p\le n}\log^{2}p\\
&\le\ Cn\log n\sum_{k\ge m+1}(1-x)^{k}\ =\ Cx^{-1}(1-x)^{m+1}n\log n
\end{align*}
for some $C>0$ follows as a consequence of \eqref{eq:log^{2}_asymp}. Hence,
$$ \limsup_{n\to\infty}\frac{A_{2}(n,m)}{n\log n}\le Cx^{-1}(1-x)^{m+1}. $$
By combining these facts, we obtain
\begin{multline*}
t\sum_{k=1}^{m}\frac{(1-x)^{k}}{k(k+1)}\ \le\ 
\liminf_{n\to\infty}\frac{\sum_{p\in\mathcal{P}\cap (\sqrt{n},nt]}(1-x)^{\lfloor nt/p\rfloor}\log^{2}p}{n\log n}\\
\le\ 
\limsup_{n\to\infty}\frac{\sum_{p\in\mathcal{P}\cap (\sqrt{n},nt]}(1-x)^{\lfloor nt/p\rfloor}\log^{2}p}{n\log n}
\le\ t\sum_{k=1}^{m}\frac{(1-x)^{k}}{k(k+1)}\ +\ Cx^{-1}(1-x)^{m+1}
\end{multline*}
for any fixed $m\in\N$. Sending $m\to\infty$ yields the assertion and completes the proof.
\end{proof}

The next lemma forms an extension of the previous one and an important ingredient in the proof of our main theorem (convergence of covariances).

\begin{lemma}\label{lem:cov_asymp}
For any fixed $x\in (0,1)$ and $0<s\le t\le 1$,
\begin{multline*}
\frac{1}{n\log n}\sum_{p\in\mathcal{P}\cap (\sqrt{n},ns]}(1-x)^{\lfloor nt/p\rfloor}\left(1-(1-x)^{\lfloor ns/p\rfloor}\right)\log^{2}p \\
\simeq\ \sum_{j\ge 1}(1-(1-x)^{j})\hspace{-.5cm}\sum_{\textstyle i\in\left(\frac{tj}{s}-1,\frac{tj}{s}+\frac{t}{s}\right)}\hspace{-.3cm}(1-x)^{i} \left(\frac{t}{i}\wedge \frac{s}{j}-\frac{t}{i+1}\vee \frac{s}{j+1}\right),
\end{multline*}
as $n\to\infty$.

\end{lemma}
\begin{proof}
We start by noting the equivalence
$$
\lfloor nt/p\rfloor=i\quad\text{and}\quad \lfloor ns/p\rfloor=j\quad\Longleftrightarrow\quad p\in \left(\frac{nt}{i+1}\vee \frac{ns}{j+1},\frac{nt}{i}\wedge \frac{ns}{j}\right],
$$
where the interval on the right can be empty. Fix $j\in\N$ and let us find all integers $i$ such that
\begin{equation}\label{eq:lem_ij_range}
\frac{t}{i+1}\vee \frac{s}{j+1}<\frac{t}{i}\wedge \frac{s}{j}.
\end{equation}

If $t/(i+1)<s/(j+1)$ or, equivalently, $i>t(j+1)/s-1$, then \eqref{eq:lem_ij_range} holds iff $i<t(j+1)/s$. If $t/(i+1)\ge s/(j+1)$ or, equivalently, $i\le t(j+1)/s-1$, then \eqref{eq:lem_ij_range} holds iff $i>tj/s-1$. Therefore, for any fixed $j\in\N$, \eqref{eq:lem_ij_range} holds iff
$$ \frac{tj}{s}-1<i<\frac{t(j+1)}{s}, $$
and this implies
\begin{align*}
&\sum_{p\in\mathcal{P}\cap (\sqrt{n},ns]}\log^{2}p (1-x)^{\lfloor nt/p\rfloor}\left(1-(1-x)^{\lfloor ns/p\rfloor}\right)\\
&\quad=\sum_{j=1}^{\lfloor\sqrt{n}s\rfloor}(1-(1-x)^{j})\sum_{i\in \left(\frac{tj}{s}-1,\frac{tj}{s}+\frac{t}{s}\right)}(1-x)^{i}\sum_{p\in\mathcal{P}\cap\left(\frac{nt}{i+1}\vee \frac{ns}{j+1},\frac{nt}{i}\wedge\frac{ns}{j}\right]}\log^{2}p.
\end{align*}
It remains to note that, for any fixed $j\in\N$ and any integer $i\in\left(\frac{tj}{s}-1,\frac{tj}{s}+\frac{t}{s}\right)$,
$$
\sum_{p\in\mathcal{P}\cap \left(\frac{nt}{i+1}\vee \frac{ns}{j+1},\frac{nt}{i}\wedge \frac{ns}{j}\right]}\log^{2}p\ \simeq\ \left(\frac{t}{i}\wedge \frac{s}{j}-\frac{t}{i+1}\vee\frac{s}{j+1}\right)n\log n,
$$
as $n\to\infty$. Now arguing in the same way as in the last part of the proof of Lemma \ref{lem:simple_sum_asymp} and noting that the series
$$
\sum_{j\ge 1}(1-(1-x)^{j})\sum_{i\in \left(\frac{tj}{s}-1,\frac{tj}{s}+\frac{t}{s}\right)}(1-x)^{i} \left(\frac{t}{i}\wedge \frac{s}{j}-\frac{t}{i+1}\vee \frac{s}{j+1}\right)
$$
converges, we obtain the assertion of the lemma.
\end{proof}

\begin{lemma}\label{lem:cheb_sum_with_remainder}
For any $x\in(0,1)$, there exists $C=C(x)>0$ such that
\begin{equation}\label{lem:sum_w_remain}
\left|\sum_{p\in\mathcal{P}:p\le t}(1-x)^{\lfloor t/p\rfloor}\log p\,-\,th(1-x)\right|\ \le\ \frac{Ct}{\log (t+2)},
\end{equation}
for all $t\ge 0$, where $h$ is as defined in Lemma \ref{lem:simple_sum_asymp}.
\end{lemma}
\begin{proof}
We may assume without loss of generality that $t\ge t_{0}$ for a fixed $t_{0}$. Recalling that $\vartheta$ denotes the first Chebyshev function defined by \eqref{eq:chebyshev_func1_def}, we can write
\begin{align*}
\sum_{p\in\mathcal{P}:p\le t}\log p\, (1-x)^{\lfloor t/p\rfloor}\ &=\ \sum_{i=1}^{\lfloor t\rfloor}(1-x)^{i} \sum_{p\in\mathcal{P}\cap (t/(i+1),t/i]}\log p\\
&=\sum_{i=1}^{\lfloor t\rfloor}(1-x)^{i} \left(\vartheta\left(\frac{t}{i}\right)-\vartheta\left(\frac{t}{i+1}\right)\right).
\end{align*}
Using the well-known inequality
$$ |\vartheta(z)-z|\ \le\ \frac{C_{1}z}{\log (z+2)}, $$
valid for all $z\ge 0$ and some $C_{1}>0$, we obtain
\begin{align*}
&\left|\sum_{p\in\mathcal{P}:p\le t}(1-x)^{\lfloor t/p\rfloor}\log p\,-\,t\sum_{i=1}^{\lfloor t\rfloor}\frac{(1-x)^{i}}{i(i+1)}\right|\\
&\hspace{2cm}\le\ \frac{2C_{1}t}{\log (t+2)}\sum_{i=1}^{\lfloor t\rfloor}(1-x)^{i}\ \le\ \frac{2(1-x)C_{1}t}{x\log (t+2)}. \end{align*}
Finally, since $h(1-x)-\sum_{i=1}^{\lfloor t\rfloor}\frac{(1-x)^{i}}{i(i+1)}=\sum_{i\ge\lfloor t\rfloor+1}\frac{(1-x)^{i}}{i(i+1)}$ satisfies the inequality
$$ t\sum_{i\ge\lfloor t\rfloor+1}\frac{(1-x)^{i}}{i(i+1)}\ \le\ \sum_{i\ge\lfloor t\rfloor+1}\frac{1}{i(i+1)}\ \le\ \frac{1}{\lfloor t\rfloor+1},
$$
we infer \eqref{lem:sum_w_remain} with $C=C(x)=\frac{2(1-x)C_{1}}{x}+1$ for all sufficiently large $t$.
\end{proof}

\bibliographystyle{amsplain}
\bibliography{StoPro}

\end{document}